\documentclass{amsart}
\usepackage{amsmath,amssymb}
\usepackage{amscd}
\usepackage{graphicx}
\usepackage[all]{xy}
\newtheorem{theorem}{Theorem}[section]
\newtheorem{corollary}{Corollary}[section]
\newtheorem{lemma}{Lemma}[section]
\newtheorem{remark}{Remark}[section]
\newtheorem{claim}{Claim}[section]
\newtheorem{proposition}{Proposition}[section]
\newtheorem{main theorem}{Main Theorem}[section]
\newtheorem*{mtTJM}{Main Theorem in [Ha2]}
\newcommand{\Hm}{H(m; (m); (\alpha))}
\newcommand{\Hn}{H(n; (n); (\beta))}
\newcommand{\pa}{\partial}
\newcommand{\Em}{E(m;(m);(\alpha))}
\newcommand{\En}{E(n;(n);(\beta))}
\newcommand{\W}{\mathcal{W}}
\newcommand{\tldW}{\widetilde{\mathcal{W}}}
\newcommand{\wt}{\widetilde}

\newcommand{\ol}{\overline}

\newcommand{\nn}{\nonumber}

\title[A classification theorem for proper mappings]
{A classification of proper holomorphic mappings 
between generalized pseudoellipsoids of different dimensions}
\author{Atsushi Hayashimoto}
\address{Atsushi Hayashimoto: 716 Tokuma, Nagano 381-8550, Japan}
\email{atsushi@nagano-nct.ac.jp}
\begin{document}
\maketitle
\begin{abstract}
We classify proper holomorphic mappings 
between generalized pseudoellipsoids of different dimensions.  
Those domains are parametrized by the exponents.  
The relations among them are also obtained. 
Main tool is the orthogonal decomposition of a CR bundle. 
Such a decomposition derives the ``variable-splitting'' of the mapping.  
\end{abstract}
\setcounter{footnote}{-1}
\footnote
{This work was supported by Grant-in-Aid for Scientific Research (C) 17K05308}
\section{Introduction} 
Let $\alpha_1, \dots, \alpha_{N-1}$ be positive integers with 
$\alpha_1, \dots, \alpha_{N-1} \geq 2$ and 
$w_j=(w^{1}_{j}, \dots, \\ w^{m_j}_{j}) \in \mathbb{C}^{m_j}, 
w=(w_1, \dots, w_N) \in \mathbb{C}^{m_1} 
\times \dots \times \mathbb{C}^{m_N}$, $m_1+\dots+m_N=m$ 
and $z \in \mathbb{C}$. 
Let $R(z, w)$ be a real analytic function defined by 
\begin{equation}
R(z, w)=|z|^2+\sum^{N-1}_{j=1}||w_j||^{2\alpha_j}+||w_N||^2-1, \nonumber
\end{equation}
where $||w_j||^2=|w^{1}_{j}|^2+\dots+|w^{m_j}_{j}|^2$ denotes the squared Euclidean norm. 

Let $E(m; m_1, \dots, m_N; \alpha_1, \dots, \alpha_{N-1}, 1)$ 
be a bounded domain 
in $\mathbb{C}^{m+1}$ with real analytic boundary defined by 
the following: 
\begin{align}
&E(m; m_1, \dots, m_N; \alpha_1, \dots, \alpha_{N-1}, 1) \nn \\
&\qquad \qquad =\{(z, w_1, \dots, w_N) \in \mathbb{C} \times \mathbb{C}^{m_1} 
\times \dots \times 
\mathbb{C}^{m_N} : R(z, w)<0\}. \nonumber 
\end{align}
We call this domain a generalized pseudoellipsoid with $N$ blocks. 
We abbreviate 
$E(m; m_1, \dots, m_N; \alpha_1, \dots, \alpha_{N-1}, 1)$ as 
$E(m; (m); (\alpha))$. 
We sometimes write $\alpha_N=1$. 

Let $\Hm$ be a pseudoconvex domain defined by 
\begin{align}
\Hm=&
\{(z, w_1, \dots, w_N) \in \mathbb{C}\times \mathbb{C}^{m_1} \times \dots 
      \times \mathbb{C}^{m_N} : \nn \\
     &\text{Im}~z>||w_1||^{2\alpha_1}+\dots+||w_{N-1}||^{2\alpha_{N-1}}
         +||w_N||^2\}.   \nonumber 
\end{align}
The generalized pseudoellipsoid $E(m; (m); (\alpha))$ is biholomorphically equivalent to 
its unbounded representation $\Hm$ 
via the mapping 
\begin{align}
\Psi & :  \Hm \ni (z, w_1, \dots, w_N)  \label{emhm}\\
& \quad \mapsto
(\dfrac{i-z}{i+z}, \dfrac{2^{1/\alpha_1}w_1}{(i+z)^{1/\alpha_1}}, \dots, 
\dfrac{2^{1/\alpha_{N-1}}w_{N-1}}{(i+z)^{1/\alpha_{N-1}}}, 
\dfrac{2w_N}{i+z}) \in E(m; (m); (\alpha)).  \nonumber 
\end{align}

We study a proper holomorphic mapping  
\begin{equation}
(\mathcal{F}, \mathcal{G})=(\mathcal{F}, \mathcal{G}_1, \dots, \mathcal{G}_N):\Em \to \En \nn
\end{equation}
or its unbounded representation 
\begin{equation}
\Phi=\wt{\Psi}^{-1} \circ (\mathcal{F}, \mathcal{G}) \circ \Psi : \Hm \to \Hn,  \nn
\end{equation}
here $\wt{\Psi}$ is a mapping between $\Hn$ and $\En$ defined analogously by  (\ref{emhm}). 
Let $\Phi=(F, G)=(F_1, G_1, \dots, G_N)$. 
If $(\mathcal{F}, \mathcal{G})$ extends holomorphically past the boundary, so does $(F, G)$. 
In this case, we use the extended mapping $(F, G)$ in two meanings without any mention: 
one is the mapping between their closures: $\ol{\Hm} \to \ol{\Hn}$  
and the other is  its restriction to the boundaries: $\pa \Hm \to \pa \Hn$.  

Let $f, g : D_1 \to D_2$ be mappings with the same source and target domains. 
We say that $f$ is equivalent to $g$  
if there exist automorphisms $\phi_1 \in \text{Aut}(D_1)$ and $\phi_2 \in \text{Aut}(D_2)$ 
such that $f=\phi_2 \circ g \circ \phi_1$ holds.  

Now we introduce a homogeneous proper holomorphic mapping $H_M(z)$ 
between balls of different dimensions. 
For a variable $z=(z_1, \dots, z_n) \in \mathbb{C}^n$, we define the tensor product by 
\begin{equation}
z \otimes z=(z_1z_1,  z_2z_1, \dots, z_nz_1, z_1z_2, \dots, z_nz_n). \nn
\end{equation} 
Namely, the tensor product of $z$ and $z$ is the mapping whose components 
are all possible products of the components of $z$. 
Repeat tensor products $M$ times and after applying a unitary transformation 
that collects the same monomials together 
and ignoring zero components, we denote by the resulted mapping $H_M(z)$, 
which has an expression:  
\begin{equation}
H_M(z)=(\dots, \sqrt{\dfrac{M!}{p_1! \dots p_n!}}z^p, \dots) \label{h}
\end{equation}
for $p=(p_1, \dots, p_n) \in (\mathbb{Z}_{\geq 0})^n, \; p_1+\dots+p_n=M$. 
D'Angelo introduced this mapping to classify proper holomorphic mappings 
between balls of different dimensions, in fact, 
\begin{equation}
H_M(z) : B^n \to B^N, 
N=\begin{pmatrix}
n+M-1 \\
M
\end{pmatrix} \nn
\end{equation}
is a proper holomorphic mapping and satisfies the equality 
$||z||^{2M}=||H_M(z)||^2$. 
For this mapping, refer the book \cite{D}.  

With the above notations and terminology, we can state our result as follows: 
\begin{main theorem}\label{main}
Let $\Em$ and $\En$ be generalized pseudoellipsoids with $N$ blocks. 
Let $(\mathcal{F}, \mathcal{G}_1, \dots, \mathcal{G}_N) : \Em \to \En$ 
be a proper holomorphic mapping that is holomorphic up to the boundary. 
Assume that all components $\mathcal{F}, \mathcal{G}_1, \dots, \mathcal{G}_N$ 
are not constant and that $m_j \geq 2, j=1, \dots, N-1$. 
Then we have the following$\mathrm{:}$ 

$\mathrm{(1)}$ There exists a permutation $\sigma$ 
of $\{1, \dots, N-1\}$ such that 
\begin{equation}
\mathcal{G}_j|_{w_{\sigma(j)}=0}=0, \quad j=1, \dots, N-1. \nn
\end{equation}

$\mathrm{(2)}$
There exist integers $M_1, \dots, M_{N-1}$ such that 
$\alpha_{\sigma(j)}=M_j \beta_j, j=1, \dots, N-1$. 

$\mathrm{(3)}$ $n_j \geq \dfrac{(M_j+m_{\sigma(j)}-1)!}{M_j ! (m_{\sigma(j)}-1)!}, \quad  
 j=1, \dots, N-1$. 

$\mathrm{(4)}$ $(\mathcal{F}, \mathcal{G}_1, \dots, \mathcal{G}_N)$ is equivalent to 
$(\widetilde{\mathcal{F}}, \widetilde{\mathcal{G}}_1, \dots, \widetilde{\mathcal{G}}_N)$ 
of the form 
\begin{align}
&\widetilde{\mathcal{F}}(z, w_1, \dots, w_N)=z, \nn \\ 
&\widetilde{\mathcal{G}}_j(z, w_1, \dots, w_N)=(H_{M_j}(w_{\sigma(j)}), 0),  \quad 
    j=1, \dots, N-1, \nn \\
&\widetilde{\mathcal{G}}_N(z, w_1, \dots, w_N)=(w_N, 0). \nn
\end{align}
\end{main theorem}

\begin{remark}
Outline of the proof is the following. 

$\mathrm{(1)}$ Decompose a CR vector bundle into a special kind of 
direct sum and the decomposition derives 
expansions of the mapping. See Sections~$\mathrm{2}$ and $\mathrm{3}$. 

$\mathrm{(2)}$ Restricting our situation to a certain variety, 
we get the relations of exponents 
$\alpha_j$ and $\beta_j$. See Section~$\mathrm{4}$. 

$\mathrm{(3)}$ Again we restrict to a certain variety to reduce 
an $N$ blocks case to a one block case.  
Then we obtain the normalization of a certain component of the mapping. 
Varying the variety to restrict, we obtain the whole normalization of the mapping.  
See Sections~$\mathrm{5}$ and $\mathrm{6}$.  
\end{remark}
If $m_j=1$, we call the domain $E(m; (1), \alpha)$ a pseudoellipsoid. 
For pseudoellipsoids and generalized pseudoellipsoids, following topics have been studied. 

(1) Given two (generalized) pseudoellipsoids, find necessary and sufficient conditions 
to exist proper holomorphic mappings between them. 
Classify the mappings between them. 
\cite{DP1}, \cite{DP2}, \cite{Ha1}, \cite{Ha2}, \cite{K4}, \cite{L}, \cite{MM}  

(2) Given any local biholomorphic mapping on a neighborhood of a boundary point 
such that the point is mapped to a boundary point, 
prove that it can be extended to a global biholomorphic mapping. 
\cite{DP2}, \cite{K3}

(3) Characterize (generalized) pseudoellipsoids by mean of their 
holomorphic automorphism groups. 
\cite{GK}, \cite{K1}, \cite{K2}, \cite{K3}, \cite{KKM}

Present article is related to (1). 
The first result on this line is done by M.~Landucci \cite{L}. 
He studied pseudoellipsoids 
\begin{equation}
E(\alpha)=\{z \in \mathbb{C}^n : \sum^{n}_{k=1}|z_k|^{2\alpha_k}<1\} \nn
\end{equation}
for $\alpha_k \in \mathbb{N}$. 
If there exists a proper holomorphic mapping $f:E(\alpha) \to E(\beta)$, 
then, after reordering the variables, 
$\alpha_j/\beta_j \in \mathbb{N}$ holds, 
and the mapping is equivalent to 
$(z_1, \dots, z_n) \mapsto (z_1^{\alpha_1/\beta_1}, \dots, z_n^{\alpha_n/\beta_n})$. 
Another proof of the relations between $\alpha_j$ and $\beta_j$ 
was given by the author 
\cite{Ha1} using the tangential Cauchy-Riemann equation.  
After that, G.~Dini-A.~S.~Primicerio \cite{DP1} proved the case of 
$\alpha_j \in \mathbb{R}_{>0}$ in Landucci's theorem. 
The determination of the mappings between generalized pseudoellipsoids 
was given by R.~Monti-D.~Morbidelli \cite{MM}. 
They studied local CR mappings, which are viewed as boundary values of 
proper holomorphic mappings, 
between boundaries of generalized pseudoellipsoids 
and decomposed them to elementary mappings, namely, any such local CR mappings are  
composite mappings of inversions, one-parameter group of dilations, 
a kind of unitary transformations, and shifts.  
G.~Dini and A.~S.~Primicerio studied more general domain,  
which is denoted by $\Sigma(m,\beta, d)$, and obtained the relations of indices. 
Let $\Sigma(m,\beta, d)$ be a domain defined by 
\begin{equation}  
\Sigma(m,\beta, d)=\{z \in \mathbb{C}^N: 
\sum_{k=1}^{n}\left(\sum_{j=s_{k-1}}^{s_k-1}|z_j|^{2d_j}\right)^{\beta_k}<1\}, \nn
\end{equation}
where 
$d=(d_1, \dots, d_N) \in \mathbb{N}^N, \; 
\beta=(\beta_1, \dots, \beta_n) \in \mathbb{N}^n$, and  
$N=m_1+\dots+m_n, s_0=1, s_k=s_{k-1}+m_k$ and $m=(m_1, \dots, m_n)$.  
If there exists a proper holomorphic mapping between 
$\Sigma(p, \alpha, c)$ and $\Sigma(m, \beta, d)$, 
then, up to reordering the variables, $p=m, \; \alpha=\beta$ and 
$c_j/d_j \in \mathbb{N}$.  

All results mentioned above are for the equi-dimensional case, namely, 
source and target domains have the same dimensions. 
On the other hand, our present research is to study the higher codimensional case. 
By codimension, we mean the difference of the source and the target dimensions. 
When we study such a case, we find some strange phenomena.  
For instance, by \cite{Do}, 
there exists a proper holomorphic mapping $f : B^{N} \to B^{N+1}$ that 
can be extended continuously to the closure $\ol{B^{N}}$ but  
can not be extended in a $C^2$ way to any open subset of the boundary of $B^N$. 
The determination of proper holomorphic mappings 
between pseudoellipsoids for small codimensional case 
was studied by P.~Ebenfelt-D.~N.~Son \cite{ES} and they obtained the explicit parametrization 
of the mappings. 
In the case of generalized pseudoellipsoids, 
under some assumptions on dimensions and nondegeneracy condition of 
a certain kind of Jacobian matrix,   
the author \cite{Ha2} obtained the conditions to exist mappings between those domains and 
gave equivalence classes of the mappings. 
The purpose of this article is to weaken those assumptions. 
The relations between author's previous results and the present results 
are discussed in the final section. 

We list some notation. 
Let $w_j=(w^{1}_{j}, \dots, w^{m_j}_{j}) \in \mathbb{C}^{m_j}, 
w=(w_1, \dots, w_N) \in \mathbb{C}^{m_1} 
\times \dots \times \mathbb{C}^{m_N}$,  
and 
$p_j=(p^{1}_{j}, \dots, p^{m_j}_{j}) \in (\mathbb{Z}_{\geq 0})^{m_j}, \; 
p=(p_1, \dots, p_N) \in (\mathbb{Z}_{\geq 0})^{m_1} 
\times \dots \times (\mathbb{Z}_{\geq 0})^{m_N}$, 
and $\alpha=(\alpha_1, \dots, \alpha_{N-1}, 1) \in \mathbb{N}^{N}$.  
\begin{itemize}
\item The multi-index notations;  
       $(w_{j})^{p_j}=(w^{1}_{j})^{p^{1}_{j}} \cdots (w^{m_j}_{j})^{p^{m_j}_{j}}$ 
       and
      $(w)^p=(w_1)^{p_1} \\ \times \cdots \times (w_N)^{p_N}$. 
\item $|||w|||^{2\alpha}=||w_1||^{2\alpha_1}+ \dots +||w_{N-1}||^{2\alpha_{N-1}}
           +||w_N||^2$. 
      $|p_j|=p^{1}_{j}+\dots+p^{m_j}_{j}$ and 
      $|p|=|p_1|+\dots+|p_N|$. 
\item The total degree of $w_j$ and $\ol{w_j}$ in a monomial $(w_j)^{p_j} \times (\ol{w_j})^{q_j}$
          is the sum $|p_j|+|q_j|$. 
\item $\pa \Omega$ is the boundary of the domain $\Omega$. 
         $\pa^{\ast}\Omega$ is the strongly pseudoconvex part of $\pa \Omega$. 
\item $\text{Re}~z=x$.  
\end{itemize}
When we need to distinguish the notation of source and 
target objects, we add `tilde' on the target objects. 

The author would like to express his sincere thanks to Professor Akio Kodama for 
reading this article and pointing out some mistakes. 
His advice was very  valuable and made some arguments clear.  

\section{decomposition of a CR vector bundle}
Following \cite{MM}, we decompose a CR vector bundle. 
We employ the usual coordinate 
$(x, w) \in \mathbb{R} \times \mathbb{C}^m$ 
on the boundary of $\Hm$, that is, 
we identify naturally $(x, w)$ with $(x+i|||w|||^{2\alpha}, w)$ 
throughout this paper. 

We define CR vector fields $L^{\lambda}_{j}$ by 
\begin{equation}
L^{\lambda}_{j}=\dfrac{\pa}{\pa w^{\lambda}_{j}}+i\alpha_j||w_j||^{2(\alpha_j-1)}
   \ol{w^{\lambda}_{j}}\dfrac{\pa}{\pa x}, \qquad 
   j=1, \dots, N, \quad \lambda=1, \dots, m_j. \nn
\end{equation}
Define the pseudohermitian structure $\vartheta$ as 
\begin{equation}
\vartheta=dx-i\sum^{N}_{j=1} \sum^{m_j}_{\mu=1}
\alpha_j ||w_j||^{2(\alpha_j-1)}
   \left(\ol{w^{\mu}_{j}}dw^{\mu}_{j}-
   w^{\mu}_{j}d\ol{w^{\mu}_{j}} \right). \nn
\end{equation}
We define  
the Levi form on $\pa \Hm$ to be a $2$-form: 
$\text{Levi}=-id\vartheta$.

We introduce the holomorphic vector fields on $\pa^{\ast}\Hm$: 
\begin{equation}
E_j=\dfrac{1}{\alpha_j}\sum^{m_j}_{\mu=1}w^{\mu}_{j}L^{\mu}_{j}, \; j=1, \dots, N,    \nn
\end{equation}
which is called a vector field of radial type. 
The vector fields $E_1, \dots, E_N$ span an $N$-dimensional 
subbundle $\mathcal{E} \subset T^{1,0} \pa^{\ast} \Hm$. 
We denote by $\mathcal{E}^{\bot}$ the orthogonal component of 
$\mathcal{E}$ in $T^{1,0}\pa^{\ast} \Hm$ with respect to the Levi form. 

Let 
$Q:T^{1,0}\pa^{\ast} \Hm \to \mathcal{E}^{\bot}$ 
be the projection 
\begin{equation} 
Q(X)=X-\sum^{N}_{j=1}\dfrac{\text{Levi}(X, \ol{E_j})}
{\text{Levi}(E_j, \overline{E_j})}E_j. \nn
\end{equation}
Let $W^{\lambda}_{j}=Q(L^{\lambda}_{j})$. 
Then, by calculation, 
it is written as 
\begin{align}
W^{\lambda}_{j}=&\dfrac{\pa}{\pa w^{\lambda}_{j}}-\sum^{m_j}_{\mu=1}
\dfrac{\ol{w^{\lambda}_{j}}w^{\mu}_{j}}{||w_j||^2}
\dfrac{\pa}{\pa w^{\mu}_{j}} \nn \\
=&\sum^{m_j}_{\mu=1}
\left\{\delta^{\lambda}_{\mu}-\dfrac{\ol{w^{\lambda}_{j}}w^{\mu}_{j}}{||w_j||^2}
\right\}L^{\mu}_{j}.  \label{wexp}
\end{align} 
Here $\delta^{\lambda}_{\mu}$ is the Kronecker's delta. 
Then the bundle $\mathcal{E}^{\bot}$ is an $(m-N)$-dimensional subbundle of 
$T^{1,0}\pa^{\ast} \Hm$ and is generated by the set of vector fields 
$\{W^{\lambda}_{j}\}$.   

Let $\mathcal{W}_j$ be the subbundle of $T^{1,0}\pa^{\ast} \Hm$ generated by 
$W^{1}_{j}, \dots, W^{m_j}_{j}$. 
Since the relation $w^{1}_{j}W^{1}_{j}+\dots+w^{m_j}_{j}W^{m_j}_{j}=0$ holds, 
the subbundle $\mathcal{W}_j$ is, in fact, generated by, say, 
$W^{1}_{j}, \dots, W^{m_j-1}_{j}$. 
Then we have 
$\mathcal{E}^{\bot}=\mathcal{W}_1 \oplus \dots \oplus \mathcal{W}_N$ 
and 
$\mathcal{W}_j \bot \mathcal{W}_k$ if $j \ne k$. 
Hence we obtain the orthogonal decomposition 
\begin{equation}
T^{1, 0}\pa^{\ast} \Hm=\mathcal{W}_1 \oplus \dots \oplus \mathcal{W}_N \oplus \mathcal{E}. 
\label{ds}
\end{equation}
We sometimes use the notation $\mathcal{E}=\mathcal{W}_{N+1}$.  
For later use, we here list the values of the Levi form as follows: 
Using the fact that 
$\theta(L^{\lambda}_{j})=0$ and 
\begin{equation}
\theta([L^{\lambda}_{j}, \ol{L^{\mu}_{j}}])=-2i\alpha_j||w_j||^{2(\alpha_j-1)}
 \Big\{\delta^{\mu}_{\lambda}+(\alpha_j-1)\dfrac{\ol{w^{\lambda}_{j}}w^{\mu}_{j}}{||w_j||^{2}}
\Bigr\}, \nn
\end{equation}
we obtain that 
\begin{equation}
\begin{split}
&\text{Levi}(W^{\lambda}_{j}, \ol{W^{\mu}_{j}})=-2\alpha_j||w_j||^{2(\alpha_j-1)}
     \dfrac{\ol{w^{\lambda}_{j}}w^{\mu}_{j}}{||w_j||^{2}},  \quad 
\text{Levi}(W^{\lambda}_{j}, \ol{W^{\mu}_{k}})=0, \; j \neq k,  \\
&\text{Levi}(W^{\lambda}_{j}, \ol{E_k})=0, \quad 
\text{Levi}(E_j, \ol{E_j})=2||w_j||^{2\alpha_j}.  \label{levical}
\end{split}
\end{equation}
Rest of this section is devoted to obtain relations among 
$\W_j$'s, $\tldW_i$'s and $\Phi_{\ast}$. 
In the following, for $j, k=1,\dots, N$, we denote by $\wt{W}^{\lambda}_{j, k}$ 
(resp.$\wt{W}^{\lambda}_{j, N+1}$) 
the $\tldW_k$-component (resp. the $\wt{\mathcal{E}}$-component) 
of $\Phi_{\ast}W^{\lambda}_{j}$. 
For a fixed $j=1, \dots, N$, also denote by $\wt{W}_{N+1, k},$ $k=1, \dots, N$ 
(resp. $\wt{W}_{N+1, N+1}$)  
the $\tldW_k$-component (resp. the $\wt{\mathcal{E}}$-component) 
of \\ $\Phi_{\ast}E_j$. 
Let $J_i$ be a subset of $\{1, \dots, N\}$ defined by 
$J_i=\{\ell \in \{1, \dots, N\} : \tldW_{\ell}\text{-} \\ \text{component of } 
\Phi_{\ast}\mathcal{W}_i \text{ is equal to zero} \}$. 
\begin{lemma}\label{indexdecomposition}
For any $j, k=1, \dots, N$ with $j \neq k$,  we have 
$J_j \cap J_k=\emptyset$. 
\end{lemma}
\begin{proof}
Take any $W^{\lambda}_{j} \in \W_j$ and $W^{\mu}_{k} \in \W_k$
for $j, k=1, \dots, N$ with $j \neq k$. 
Let us denote by 
\begin{align}
&\Phi_{\ast}W^{\lambda}_{j}=\wt{W}_{j,1}^{\lambda}+ \dots+
        \wt{W}_{j, N}^{\lambda}+\wt{W}_{j, N+1}^{\lambda},  \label{phiwj}\\
&\Phi_{\ast}W^{\mu}_{k}=\wt{W}_{k, 1}^{\mu}+ \dots+
        \wt{W}_{k, N}^{\mu}+\wt{W}_{k, N+1}^{\mu}. \label{phiwk}
\end{align}
We calculate the Levi form 
$\widetilde{\text{Levi}}(\Phi_{\ast}W^{\lambda}_{j}, \ol{\Phi_{\ast}W^{\mu}_{k}})$ 
in two ways.  
Since the mapping $\Phi$ is CR, there exists a positive function $f$ such that 
$\Phi^{\ast}\widetilde{\vartheta}=f\vartheta$. 
Hence we have 
\begin{equation}
\widetilde{\text{Levi}}(\Phi_{\ast}W^{\lambda}_{j}, \ol{\Phi_{\ast}W^{\mu}_{k}})
=f\text{Levi}(W^{\lambda}_{j}, \ol{W^{\mu}_{k}})=0.  \nn
\end{equation}
On the other hand, by using (\ref{phiwj}) and (\ref{phiwk}), 
it is expanded as follows: 
\begin{align}
&\widetilde{\text{Levi}}(\Phi_{\ast}W^{\lambda}_{j}, \ol{\Phi_{\ast}W^{\mu}_{k}}) 
    \label{leviexp} \\
&=\widetilde{\text{Levi}}(\wt{W}^{\lambda}_{j, 1}, \ol{\wt{W}^{\mu}_{k, 1}})+ \dots+ 
    \widetilde{\text{Levi}}(\wt{W}^{\lambda}_{j, N}, \ol{\wt{W}^{\mu}_{k, N}})+ 
   \widetilde{\text{Levi}}(\wt{W}^{\lambda}_{j, N+1}, \ol{\wt{W}^{\mu}_{k, N+1}})
     \nn \\
   & \quad 
   +\wt{\text{Levi}}(\wt{W}^{\lambda}_{j, 1}, \ol{\wt{W}^{\mu}_{k, N+1}})+\dots
  +\wt{\text{Levi}}(\wt{W}^{\lambda}_{j, N}, \ol{\wt{W}^{\mu}_{k, N+1}}) \nn  \\ 
& \quad 
    +\wt{\text{Levi}}(\wt{W}^{\lambda}_{j, N+1}, \ol{\wt{W}^{\mu}_{k, 1}})+\dots
    +\widetilde{\text{Levi}}(\wt{W}^{\lambda}_{j, N+1}, \ol{\wt{W}^{\mu}_{k, N}}).   \nn
\end{align}
For $i=1, \dots, N$, let us denote by 
\begin{align}
&\wt{W}^{\lambda}_{j, i}=a^{1}_{i}\wt{W}^{1}_{i}+\dots+
     a^{n_{i}-1}_{i}\wt{W}^{n_{i}-1}_{i},  \quad 
\wt{W}^{\mu}_{k, i}=b^{1}_{i}\wt{W}^{1}_{i}+\dots+
     b^{n_{i}-1}_{i}\wt{W}^{n_{i}-1}_{i}, \nn  \\
&\wt{W}^{\lambda}_{j, N+1}=A_1\wt{E}_1+\dots+A_N\wt{E}_N,  \quad 
\wt{W}^{\mu}_{k, N+1}=B_1\wt{E}_1+\dots+B_N\wt{E}_N,  \nn 
\end{align}
then it follows from (\ref{levical}) that each term on the right hand side of (\ref{leviexp}) 
is calculated as follows: 
\begin{align}
&\wt{\text{Levi}}(\wt{W}^{\lambda}_{j, i}, \ol{\wt{W}^{\mu}_{k, i}})
=-2\beta_{i}||\wt{w}_{i}||^{2(\beta_{i}-1)}
\sum_{\nu_1, \nu_2=1}^{n_{i}-1}a^{\nu_1}_{i}\ol{b^{\nu_2}_{i}}
\dfrac{\ol{\wt{w}^{\nu_1}_{i}} \wt{w}^{\nu_2}_{i}}
      {||\widetilde{w}_{i}||^2}, \; i=1, \dots, N,   \nn  \\
&\wt{\text{Levi}}(\wt{W}^{\lambda}_{j, N+1}, \ol{\wt{W}^{\mu}_{k, N+1}})
=2(A_1\ol{B_1}||\wt{w}_1||^{2\beta_1}+\dots+A_N\ol{B_N}||\wt{w}_N||^{2}),   \nn \\
&\wt{\text{Levi}}(\wt{W}^{\lambda}_{j, i}, \ol{\wt{W}^{\mu}_{k, N+1}})=
\widetilde{\text{Levi}}(\wt{W}^{\lambda}_{j, N+1}, \ol{\wt{W}^{\mu}_{k, i}})=0, \; i=1, \dots, N.  \nn
\end{align}
Hence we obtain 
\begin{equation}
-\sum^{N}_{i=1}2\beta_{i}||\wt{w}_{i}||^{2(\beta_{i}-1)}
\sum_{\nu_1, \nu_2=1}^{n_{i}-1}a^{\nu_1}_{i}\ol{b^{\nu_2}_{i}}
\dfrac{\ol{\wt{w}^{\nu_1}_{i}} \wt{w}^{\nu_2}_{i}}
      {||\widetilde{w}_{i}||^2}
+2\sum^{N}_{i=1}A_{i}\ol{B_i}||\wt{w}_i||^{2\beta_i}=0, \nn
\end{equation}
which means that for each $i=1, \dots, N$, 
\begin{align}
&a^{1}_{i}\ol{b^{1}_{i}}=\dots=a^{1}_{i}\ol{b^{n_{i}-1}_{i}}=0, \nn\\
& \qquad \vdots \nn \\
&a^{n_{i}-1}_{i}\ol{b^{1}_{i}}=\dots=a^{n_{i}-1}_{i}\ol{b^{n_{i}-1}_{i}}
=0.  \nn
\end{align}
Take any $i=1, \dots, N$ and fix it. 
If $a^{\lambda}_{i}=0$ for any $\lambda=1, \dots, n_{i}-1$, then 
$\wt{W}^{\lambda}_{j, i}=0$; and hence, the $\tldW_i$-component of $\Phi_{\ast}W^{\lambda}_{j}$ 
is equal to zero. 
If there exists $\lambda=1, \dots, n_{i}-1$ such that 
$a^{\lambda}_{i}\neq0$, then 
$b^{1}_{i}=\dots=b^{n_{i}-1}_{i}=0$ and 
this means that 
the $\tldW_i$-component of 
$\Phi_{\ast}W^{\mu}_{k}$ is equal to zero. 
Hence both $\Phi_{\ast}W^{\lambda}_{j}$ and $\Phi_{\ast}W^{\mu}_{k}$ do not have 
the $\tldW_i$-component simultaneously. 
\end{proof}
The same conclusion holds for 
$\Phi_{\ast}\mathcal{E}$ and $\Phi_{\ast}\mathcal{W}_j$. 
\begin{lemma}
For any $j=1, \dots, N$, we have $J_j \cap J_{N+1}=\emptyset$. 
\end{lemma}
\begin{proof}
Take any $k=1, \dots, N$ and fix it. 
Let us denote by 
\begin{align}
&\Phi_{\ast}E_k=\wt{W}_{N+1, 1}+\dots+\wt{W}_{N+1, N}+\wt{W}_{N+1, N+1}, \nn \\
&\Phi_{\ast}W^{\lambda}_{j}=\wt{W}^{\lambda}_{j, 1}+\dots+\wt{W}^{\lambda}_{j, N}
+\wt{W}^{\lambda}_{j, N+1}.  \nn  
\end{align}
Lemma follows by the same argument as in Lemma~{\ref{indexdecomposition}}. 
\end{proof}

\begin{lemma}\label{nowe}
There does not exist 
$W^{\lambda}_{j} \in \mathcal{W}_j, \;  j=1, \dots, N $ such that 
$\Phi_{\ast}W^{\lambda}_{j} \in \wt{\mathcal{E}}$.  
\end{lemma} 

\begin{proof}
We use the expansion of $G^{\mu}_{\ell}$: 
\begin{equation}
G^{\mu}_{\ell}(x, w)=\sum_{|p|+q \geq 0}b^{\mu}_{\ell, p, q}(w)^p
   (x+i|||w|||^{2\alpha})^q.  \nn
\end{equation} 
Suppose that there exists $W^{\lambda}_{j} \in \mathcal{W}_j$ 
such that $\Phi_{\ast}W^{\lambda}_{j} \in \wt{\mathcal{E}}$ 
and express it as 
$\Phi_{\ast}W^{\lambda}_{j}=A^{\lambda}_1\wt{E}_1+\dots+A^{\lambda}_N\wt{E}_N$.  
Take $\ell=1, \dots, N$ such that $A^{\lambda}_{\ell} \neq 0$. 
Apply $\wt{w}^{\mu}_{\ell}=G^{\mu}_{\ell}$ to this relation to get the following:  
\begin{align}
\dfrac{A^{\lambda}_{\ell}}{\beta_{\ell}}G^{\mu}_{\ell}=
W^{\lambda}_{j}G^{\mu}_{\ell}
&=\sum^{m_j}_{\nu=1}\left(\delta^{\nu}_{\lambda}-\dfrac{\ol{w^{\lambda}_{j}}w^{\nu}_{j}}{||w_j||^2}\right) 
   L^{\nu}_{j}G^{\mu}_{\ell}   \label{wg}  \\
&=\sum_{\substack{|p|+q \geq 0 \\ p^{\lambda}_{j}\geq 1}}
b^{\mu}_{\ell, p, q}\dfrac{\pa (w)^{p}}{\pa w^{\lambda}_{j}}(x+i|||w|||^{2\alpha})^q  \nn \\
& \qquad -\dfrac{\ol{w^{\lambda}_{j}}}{||w_j||^2}
\sum^{m_j}_{\nu=1}\sum_{\substack{|p|+q \geq 0 \\ p^{\nu}_j\geq 1}}
b^{\mu}_{\ell, p, q}p^{\nu}_j(w)^p
(x+i|||w|||^{2\alpha})^q. \nn 
\end{align}
For convenience, we want to rewrite the last summation in (\ref{wg}) as follows. 
To this end, let us introduce a subset $P_k$ of $\mathbb{Z}^{m_{j}}$ defined by 
\begin{align}
P_k=\{(p^{1}_{j}, \dots, p^{m_j}_j) \in \mathbb{Z}^{m_j} : \;
&k \; \text{components among} \; p^{1}_{j}, \dots, p^{m_j}_{j} \; 
\text{are positive and }  \nn \\
&\text{the rest are zero} \}.   \nn
\end{align} 
Take any $p^0_j \in P_{k}$ such that 
its $\lambda_1$-th, $\dots, \lambda_{k}$-th 
components are positive and the others are zero. 
Then, among $m_j$ summations in the second term on the far right hand side of (\ref{wg}), 
\begin{equation}
\sum_{\substack{|p|+q \geq 0 \\ p^{1}_j\geq 1}}
b^{\mu}_{\ell, p, q}p^{1}_j(w)^p
(x+i|||w|||^{2\alpha})^q, \dots, 
\sum_{\substack{|p|+q \geq 0 \\ p^{m_j}_j\geq 1}}
b^{\mu}_{\ell, p, q}p^{m_j}_j(w)^p
(x+i|||w|||^{2\alpha})^q, \nn
\end{equation}
the term of the form: 
\begin{equation}
\sum_{\substack{|p|+q \geq 0 \\ p_j=p^{0}_{j}}}  \nn
b^{\mu}_{\ell, p, q}p^{\nu}_j(w)^p
(x+i|||w|||^{2\alpha})^q
\end{equation}
comes only from 
the $\lambda_1$-th, $\dots$, the $\lambda_k$-th summations. 
Thus (\ref{wg}) can be rewritten as  
\begin{align}
\dfrac{A^{\lambda}_{\ell}}{\beta_{\ell}}G^{\mu}_{\ell}=
&\sum_{\substack{|p|+q \geq 0 \\ p^{\lambda}_{j}\geq 1}}
b^{\mu}_{\ell, p, q}\dfrac{\pa (w)^{p}}{\pa w^{\lambda}_{j}}(x+i|||w|||^{2\alpha})^q  \nn \\
&\qquad -\dfrac{\ol{w^{\lambda}_{j}}}{||w_j||^2}
\Bigl\{m_j\sum_{\substack{|p|+q \geq 0 \\ p_j \in P_{m_j}}}
b^{\mu}_{\ell, p, q}|p_j|(w)^p (x+i|||w|||^{2\alpha})^q \nn \\
& \qquad \qquad +\dots+
(m_j-k)
\sum_{\substack{|p|+q\geq 0 \\ p_j \in P_{m_j-k}}}
b^{\mu}_{\ell, p, q}|p_j|(w)^p(x+i|||w|||^{2\alpha})^q \nn \\
&\qquad \qquad +\dots+
\sum_{\substack{|p|+q \geq 0 \\ p_j \in P_1}}
b^{\mu}_{\ell, p, q}|p_j|(w)^p(x+i|||w|||^{2\alpha})^q \Bigr\}.  \nn
\end{align} 
Hence comparing the $\ol{w^{\lambda}_{j}}$ terms and the $(w)^p$ terms with 
$p_j \in P_{m_j-k}$, 
we have 
\begin{equation}
\sum_{\substack{|p|+q \geq 0 \\ p_j \in P_{m_j-k}}}b^{\mu}_{\ell, p, q}|p_j|(w)^p
(x+i|||w|||^{2\alpha})^q=0, \quad  k=0, \dots, m_j-1 \nn
\end{equation}
and thus $G^{\mu}_{\ell}$ has an expansion 
\begin{equation}
G^{\mu}_{\ell}(x, w)=\sum_{\substack{|p|+q \geq 0 \\ p_j=0}}b^{\mu}_{\ell, p, q}(w)^p
   (x+i|||w|||^{2\alpha})^q.  \nn
\end{equation} 
This means that the variables $w_j$ do not appear in $(w)^p$. 
Therefore again by the relation (\ref{wg}) and by $A^{\lambda}_{\ell} \neq 0$, 
we conclude that $G^{\mu}_{\ell}=0$ 
for any $\mu=1, \dots n_{\ell}$. This contradicts the assumption on the main theorem.   
\end{proof}

From the proof of this lemma, the following corollary holds, which will be used 
in the next argument.  
\begin{corollary}\label{f0}
Let $f$ be a CR function on $\pa \Hm$. 
Assume that there exist $W^{\lambda}_{j} \in \mathcal{W}_j$ and a constant $A$ 
such that $W^{\lambda}_{j}f=Af$. 
If $A\neq 0$, then $f\equiv0$ and if $A=0$, then 
$f$ has an expansion: 
\begin{equation}
f=\sum_{\substack{|p|+q \geq 0 \\ p_j=0}} a_{p, q}(w)^p(x+i|||w|||^{2\alpha})^q. \nn
\end{equation}
\end{corollary}

From Lemma~{\ref{indexdecomposition}} to Lemma~{\ref{nowe}}, 
we have proved that there exists a permutation $\sigma$ of $\{1, \dots, N\}$ 
such that 
$\Phi_{\ast}\mathcal{W}_{\sigma(j)} \subset \tldW_j \oplus \wt{\mathcal{E}}$. 
Here we assert that  
$\Phi_{\ast}\mathcal{W}_{\sigma(j)} \subset \tldW_j$ for all $j=1, \dots, N$. 
We now verify this only in the case where $j=1$, since the verification in the general case 
is almost identical. 
So, take any $W^{\lambda}_{\sigma(1)} \in \mathcal{W}_{\sigma(1)}$ and write 
\begin{equation}
\Phi_{\ast}W^{\lambda}_{\sigma(1)}=a^{\lambda}_1\wt{W}^{1}_{1}+\dots+a^{\lambda}_{n_1-1}\wt{W}^{n_1-1}_{1}
+A^{\lambda}_1\wt{E}_1+\dots+A^{\lambda}_N\wt{E}_N.  \label{wwe}
\end{equation}
Assume that there exist $\lambda=1, \dots, m_{\sigma(1)}$ and $\ell=2, \dots, N$ 
such that $A^{\lambda}_{\ell} \neq 0$. 
Apply $\wt{w}^{\mu}_{\ell}=G^{\mu}_{\ell}$ to (\ref{wwe}) to get 
$W^{\lambda}_{\sigma(1)}G^{\mu}_{\ell}=(A^{\lambda}_{\ell}/\beta_{\ell})G^{\mu}_{\ell}$. 
Then by Corollary~{\ref{f0}}, $G^{\mu}_{\ell}=0$ for any $\mu=1, \dots, n_{\ell}$, 
which is a contradiction. 
Thus $A^{\lambda}_{\ell}=0$ for any  $\lambda=1, \dots, m_{\sigma(1)}$ and $\ell=2, \dots, N$. 
Among $A^{1}_{1}, \dots, A^{m_{\sigma(1)}}_{1}$, 
assume that $A^{\lambda_1}_{1}, \dots, A^{\lambda_i}_{1} \neq 0$ 
and the others are zero. 
Apply $\wt{x}=(F+\ol{F})/2$ to (\ref{wwe}) to get 
$(1/2)W^{\lambda}_{\sigma(1)}F=iA^{\lambda}_{1}||G_1||^{2\beta_1}$. 
Since $F$ has an expansion
\begin{equation}
F=\sum_{|p|+q \geq 0}a_{p, q}(w)^p(x+i|||w|||^{2\alpha})^q,   \label{fwx}
\end{equation}
we get 
\begin{align}
&\dfrac{1}{2}\sum_{\substack{|p|+q \geq 0 \\ p^{\lambda}_{\sigma(1)} \geq 1}}a_{p, q}
\dfrac{\pa (w)^p}{\pa w^{\lambda}_{\sigma(1)}}(x+i|||w|||^{2\alpha})^q  \label{anylambda} \\
&-\dfrac{1}{2}\dfrac{\ol{w^{\lambda}_{\sigma(1)}}}{||w_{\sigma(1)}||^2}
\sum^{m_{\sigma(1)}}_{\nu=1}
\sum_{\substack{|p|+q \geq 0 \\ p^{\nu}_{\sigma(1)}\geq 1}}a_{p, q}p^{\nu}_{\sigma(1)}
(w)^p(x+i|||w|||^{2\alpha})^q  \nn \\
&=iA^{\lambda}_{1}||G_1||^{2\beta_1}  \nn
\end{align}
for any $\lambda=1, \dots, m_{\sigma(1)}$. 
Multiply (\ref{anylambda}) by $w^{\lambda}_{\sigma(1)}$ and add them from 
$\lambda=1$ to $\lambda=m_{\sigma(1)}$.  
Then we have 
\begin{equation}
0=i||G_1||^{2\beta_1}\{A^{\lambda_1}_{1}w^{\lambda_1}_{\sigma(1)}+\dots+
             A^{\lambda_i}_{1}w^{\lambda_i}_{\sigma(1)}\}  \nn
\end{equation}
and this means that $G_1=0$. This is a contradiction. 
Thus $A^{\lambda}_{1}=0$ for all $\lambda$. 
Hence, the $\wt{\mathcal{E}}$-component of $\Phi_{\ast}W^{\lambda}_{\sigma(1)}$ 
is equal to zero and 
$\Phi_{\ast}\mathcal{W}_{\sigma(1)} \subset \tldW_1$, as asserted. 
By the same argument as above, it can be seen that 
$\Phi_{\ast}\mathcal{E} \subset \wt{\mathcal{E}}$. 
Summarizing the above, we obtain the following: 
\begin{proposition}\label{sigmae}
There exists a permutation $\sigma$ of $\{1, \dots, N\}$ such that 
$\Phi_{\ast}\mathcal{W}_{\sigma(j)} \subset \tldW_j$ for all $j=1, \dots, N$; and 
$\Phi_{\ast}\mathcal{E} \subset \wt{\mathcal{E}}$. 
\end{proposition}

\section{Expansions of $F$ and $G$}
Using Proposition~{\ref{sigmae}}, we can characterize the expansions of 
$F$ and $G^{\lambda}_{j}$. 

\begin{lemma}\label{fgex}
Up to an automorphism of $\Hn$, 
the mapping $(F, G)$ has the following expansions  on $\pa \Hm$: 
\begin{align}
&F(x, w)=x+i|||w|||^{2\alpha},    \label{fx} \\
&G^{\mu}_{j}(w)
=\sum_{|p_{\sigma(j)}| \geq 1}b^{\mu}_{j, p_{\sigma(j)}}(w_{\sigma(j)})^{p_{\sigma(j)}}, \; j=1, \dots, N. 
\label{gw}
\end{align}
\end{lemma}
\begin{proof}
\begin{claim}\label{claimg}
$G^{\mu}_{j}$ is expanded as 
\begin{equation}
G^{\mu}_{j}(w)
=\sum_{|p_{\sigma(j)}| \geq 0}b^{\mu}_{j, p_{\sigma(j)}}(w_{\sigma(j)})^{p_{\sigma(j)}}.  
\label{gex1st}
\end{equation}
\end{claim}
Take any $j=1, \dots, N$ and fix it. 
Since $\Phi_{\ast}W^{\lambda}_{\sigma(i)} \in \tldW_i$, 
by applying $\wt{w}^{\mu}_{j}=G^{\mu}_{j}$ to 
$\Phi_{\ast}W^{\lambda}_{\sigma(i)}, i=1, \dots, N, \lambda=1, \dots, m_{\sigma(i)}$ with $i \neq j$, 
we get the relations: 
\begin{align}
&W^{1}_{\sigma(1)}G^{\mu}_{j}=0, \dots, W^{m_{\sigma(1)}}_{\sigma(1)}G^{\mu}_{j}=0, 
\nn \\
& \qquad \vdots   \nn \\
&W^{1}_{\sigma(N)}G^{\mu}_{j}=0, \dots, W^{m_{\sigma(N)}}_{\sigma(N)}G^{\mu}_{j}=0. 
\nn 
\end{align}
Note that 
$W^{1}_{\sigma(j)}G^{\mu}_{j}=0, \dots, W^{m_{\sigma(j)}}_{\sigma(j)}G^{\mu}_{j}=0$
are not contained in these relations. 
It follows from Corollary~{\ref{f0}} that we obtain the expansion: 
\begin{equation}
G^{\mu}_{j}(x, w)
=\sum_{|p_{\sigma(j)}|+q \geq 0}b^{\mu}_{j, p_{\sigma(j)}, q}
    (w_{\sigma(j)})^{p_{\sigma(j)}}(x+i|||w|||^{2\alpha})^q.  \nn
\end{equation}
We apply $\wt{w}^{\mu}_{j}=G^{\mu}_{j}$ to 
$\Phi_{\ast}E_k=A^k_1\wt{E}_1+\dots+A^k_N\wt{E}_N$ 
with $k \neq \sigma(j)$. 
Assume that $A^{k}_{j} \neq 0$ for all $k=1, \dots, N$ and $k \neq \sigma(j)$. 
Then we get the relation 
$E_kG^{\mu}_{j}=A^k_j\wt{E}_j\wt{w}^{\mu}_{j}=(A^k_j/\beta_j)G^{\mu}_{j}$. 
This relation is rewritten as 
\begin{equation}
2i||w_k||^{2\alpha_k}\sum_{\substack{|p_{\sigma(j)}|+q \geq 0 \\ q \geq 1}}b^{\mu}_{j, p_{\sigma(j), q}} 
(w_{\sigma(j)})^{p_{\sigma(j)}}q(x+i|||w|||^{2\alpha})^{q-1}
=\dfrac{A^k_j}{\beta_j}G^{\mu}_{j}.    \nn
\end{equation}
By comparing the terms with $||w_k||^{2\alpha_k}$, 
the both sides must be zero. 
Since $G^{\mu}_{j}$ does not vanish identically, $A^k_j=0$. 
This is a contradiction. 
Thus there exists $k=1, \dots, N$ such that $k \neq \sigma(j)$ 
and $A^{k}_{j}=0$. 
Take such $k$ and apply $\wt{w}^{\mu}_{j}=G^{\mu}_{j}$ to 
$\Phi_{\ast}E_k=A^k_1\wt{E}_1+\dots+A^k_N\wt{E}_N$ 
to get the relation $E_kG^{\mu}_{j}=0$. 
This is rewritten as 
\begin{equation}
2i||w_k||^{2\alpha_k}\sum_{\substack{|p_{\sigma(j)}|+q \geq 0 \\ q \geq 1}}b^{\mu}_{j, p_{\sigma(j), q}} 
(w_{\sigma(j)})^{p_{\sigma(j)}}q(x+i|||w|||^{2\alpha})^{q-1}=0.  \nn
\end{equation}
This means that the coefficients $b^{\mu}_{j, p_{\sigma(j)}, q}$ vanish for 
$|p_{\sigma(j)}|+q \geq 0, q \geq 1$. 
Now we conclude that 
$G^{\mu}_{j}$ has an expansion of the form (\ref{gex1st}). 

\vspace{1\baselineskip}

Let us denote by 
$\Phi_{\ast}E_{\sigma(1)}=A^{\sigma(1)}_{1}\wt{E}_1+\dots+A^{\sigma(1)}_{N}\wt{E}_N$. 
Apply $\wt{w}^{\mu}_{j}=G^{\mu}_{j}$ with $j \neq 1$ to get 
$(A^{\sigma(1)}_{j}/\beta_j)G^{\mu}_{j}=0$ and this means that 
$A^{\sigma(1)}_{j}=0$ for $j \neq 1$. 
By the same reason, we have $A^{\sigma(k)}_{j}=0$ for $j\neq k$, namely, we have 
$\Phi_{\ast}E_{\sigma(j)}=A^{\sigma(j)}_j\wt{E}_j$, which will be used in the next claim. 
\begin{claim}  
$F$ is expanded as 
\begin{equation}
F(x, w)=a_0+a_1(x+i|||w|||^{2\alpha}). \label{fex1st} 
\end{equation}
\end{claim}
First of all we apply $\wt{x}=(F+\ol{F})/2$ to 
$\Phi_{\ast}W^{\lambda}_{\sigma(1)} \in \tldW_1, \dots, \Phi_{\ast}W^{\lambda}_{\sigma(N)} \in \tldW_N$ 
to get 
\begin{align}
&W^{1}_{\sigma(1)}F=0, \dots, W^{m_{\sigma(1)}}_{\sigma(1)}F=0, 
\nn \\
& \qquad \vdots   \nn \\
&W^{1}_{\sigma(N)}F=0, \dots, W^{m_{\sigma(N)}}_{\sigma(N)}F=0. 
\nn 
\end{align}
By the same argument as in Claim~{\ref{claimg}}, $F$ is a function of variable $x+i|||w|||^{2\alpha}$. 
Next applying  $\wt{x}=(F+\ol{F})/2$ to 
$\Phi_{\ast}E_{\sigma(j)}=A^{\sigma(j)}_j\wt{E}_j$
and using the expansion (\ref{fwx}), we get 
\begin{equation}
||w_{\sigma(j)}||^{2\alpha_{\sigma(j)}}\sum_{q \geq 1}a_{p, q}q(x+i|||w|||^{2\alpha})^{q-1}
=A^{\sigma(j)}_j||G_j||^{2\beta_j}. \label{fex2}
\end{equation} 
Since the right hand side does not contain the variable $x$, we obtain $a_{p, q}=0$ for $q \geq 2$. 
This leads to an expansion of the form (\ref{fex1st}). 

\vspace{1\baselineskip}

Go back to the proof of lemma. 
The equation (\ref{fex2}) is reduced to 
$||w_{\sigma(j)}||^{2\alpha_{\sigma(j)}}a_1 \\ =A^{\sigma(j)}_j||G_j||^{2\beta_j}$.  
This means that $G^{\mu}_{j}$ does not have a constant term, 
which implies the desired expansion of (\ref{gw}). 
On the boundary $\pa \Hm$, we have 
\begin{equation}
\text{Im} \; (a_0+a_1(x+i|||w|||^{2\alpha}))=|||G|||^{2\beta}. 
\end{equation}
Since the right hand side does not have the constant term and the terms containing $x$, 
we obtain that $a_0$ and $a_1$ are real and also obtain that $a_1 >0$.  
Since the mapping 
\begin{equation}
(\wt{z}, \wt{w}) \mapsto (\dfrac{\wt{z}-a_0}{a_1}, 
\dfrac{\wt{w}_1}{a^{1/2\beta_{\sigma(1)}}_1}, \dots, 
\dfrac{\wt{w}_{N-1}}{a^{1/2\beta_{\sigma(N-1)}}_1}, \dfrac{\wt{w}_N}{a^{1/2}_1})
\end{equation}
is an automorphism of $\Hn$. 
It follows that $F$ has an desired expansion (\ref{fx}). 
This automorphism does not change the form of (\ref{gw}). 
\end{proof}

\section{The relations between exponents}
Making use of the expansions of $F$ and $G^{\lambda}_{j}$, 
we will get the relations between exponents $\alpha_j$ and $\beta_j$. 
It follows from Lemma~{\ref{fgex}} that we have $|||w|||^{2\alpha}=|||G|||^{2\beta}$ 
on the boundary of $\Hm$. 
Take $j$ with $\sigma(j)=N$. 
We restrict $|||w|||^{2\alpha}=|||G|||^{2\beta}$ to 
the variety 
$w_1=\dots=w_{N-1}=0$. 
Then it is reduced to 
$||w_N||^{2}=||G_j(w_{N})||^{2\beta_j}$. 
Substitute the expansion of $G^{\lambda}_j(w_{N})$ into this and 
compare the degree of $w_N$. 
Then we obtain $\beta_j=1$, which implies that $j=N$. 
Hence we get $\sigma(N)=N$. 
Thus, from now on, we consider $\sigma$ as a permutation of $\{1, \dots, N-1\}$. 

Now we obtain the relations between the exponents $\alpha$ and $\beta$. 
\begin{lemma}\label{alphabeta}
There exist $M_1, \dots, M_{N-1} \in \mathbb{N}$ such that 
$\alpha_{\sigma(j)}=M_{j}\beta_j$ for $j=1, \dots, N-1$. 
\end{lemma}  

\begin{proof}
Once we get the relation between $\alpha_{\sigma(1)}$ and $\beta_1$, the relations between 
$\alpha_{\sigma(j)}$ and $\beta_j$ for $2 \leq j \leq N-1$ 
are obtained by the same argument. 
So, it suffices to show the case $j=1$. 
First we restrict the expansions (\ref{fx}) and (\ref{gw}) to the variety 
$w_{\sigma(2)}=\dots=w_{\sigma(N-1)}=w_N=0$
and obtain the relation on the boundary $\pa \Hm$: 
\begin{equation}
||w_{\sigma(1)}||^{2\alpha_{\sigma(1)}}=
||\sum_{|p_{\sigma(1)}| \geq 1}b_{1, p_{\sigma(1)}}
(w_{\sigma(1)})^{p_{\sigma(1)}}||^{2\beta_1}. \label{ab}
\end{equation}

Case 1: $\alpha_{\sigma(1)}<\beta_1$. 
Comparing the minimal total degree of 
$w_{\sigma(1)}$ and $\ol{w_{\sigma(1)}}$ in the both sides, 
we reach a contradiction. 

Case 2: $\alpha_{\sigma(1)} \geq \beta_1$. 
Assume that there does not exist $M_1 \in \mathbb{N}$ as in the lemma. 
Take $M \in \mathbb{N}$  such that 
$M\beta_1 < \alpha_{\sigma(1)} <(M+1)\beta_1$. 
The terms of $w_{\sigma(1)}$ and $\ol{w_{\sigma(1)}}$ which have total degree  
less than or equal to $2M\beta_1$ on the right hand side are zero. 
Hence the equation (\ref{ab}) becomes 
\begin{equation}
||w_{\sigma(1)}||^{2\alpha_{\sigma(1)}}
=||\sum_{|p_{\sigma(1)}| \geq M+1}
b_{1; p_{\sigma(1)}}
    (w_{\sigma(1)})^{p_{\sigma(1)}}||^{2\beta_1}. 
         \nn
\end{equation}
Comparing the minimal total degree of $w_{\sigma(1)}$ and $\ol{w_{\sigma(1)}}$ 
in the both sides, we reach a contradiction. 
Hence we obtain $\alpha_{\sigma(1)}=M_1\beta_1$ for some $M_1 \in \mathbb{N}$. 
This completes the proof. 
\end{proof}

\section{Normalization of the mapping}
Substituting $\alpha_{\sigma(1)}=M_1 \beta_1$ to (\ref{ab}) and 
comparing the total degree of 
$w_{\sigma(1)}$ and $\ol{w_{\sigma(1)}}$, 
we conclude that $G^{\mu}_{1}$ is of the form: 
\begin{equation}
G^{\mu}_{1}(w_{\sigma(1)})=\sum_{|p_{\sigma(1)}|=M_1}
   b^{\mu}_{1;p_{\sigma(1)}}(w_{\sigma(1)})^{p_{\sigma(1)}}.  \nn
\end{equation}
Analogously, we conclude that $G^{\mu}_{j}, j=1, \dots, N,$ have the following expansions:
\begin{align}
&G^{\mu}_{j}(w_{\sigma(j)})=\sum_{|p_{\sigma(j)}|=M_j}
 b^{\mu}_{j, p_{\sigma(j)}}(w_{\sigma(j)})^{p_{\sigma(j)}}, 
\; j=1, \dots, N-1, \; \mu=1, \dots, n_j, \label{gm} \\
&G^{\mu}_{N}(w_N)=\sum_{|p_N|=1}b^{\mu}_{N, p_N}(w_N)^{p_N},  
    \;  \mu=1, \dots, n_N. \nn
\end{align} 

We give a normalization of one block case by 
restricting $\Phi=(F, G)$ to the variety $w_{\sigma(2)}=\dots=w_{\sigma(N-1)}=w_N=0$.  
This restriction does not influence to $G_1$. 
The reduced mapping is now 
\begin{equation}
(F, G_1):\{\text{Im}z=||w_{\sigma(1)}||^{2\alpha_{\sigma(1)}}\} \to 
\{\text{Im}\wt{z}=||\wt{w}_1||^{2\beta_1}\}. 
\end{equation}

\begin{lemma}\label{dangelonormal}
Let $\alpha_{\sigma(1)}=M_1\beta_1$. Then we have the following: 

$\mathrm{(a)}$ \;   
$n_1 \geq \dfrac{(M_1+m_{\sigma(1)}-1)!}{M_1!(m_{\sigma(1)}-1)!}; $ and 

$\mathrm{(b)}$ \; 
$(F, G_1)$ is equivalent to 
$(z, H_{M_1}(w_{\sigma(1)}), 0)$. 
\end{lemma}
\begin{proof}
We introduce two notations: 
\begin{equation}
T_{M, m}=\dfrac{(M+m-1)!}{M!(m-1)!}=
\begin{pmatrix}
M+m-1 \\
m-1
\end{pmatrix}, \quad 
C_{M, p}=\sqrt{\dfrac{M!}{p^{1}! \dots p^{m}!}} \nn
\end{equation}
for $M, m \in \mathbb{N}$ and $p=(p^{1}, \dots, p^{m}) \in \mathbb{N}^m$. 
Note that $T_{M, m}$ is the number of solutions of $x_1+\dots+x_{m}=M$ 
for $x_1, \dots, x_{m} \in \mathbb{Z}_{\geq 0}$. 
Substituting (\ref{gm}) and $a_{\sigma(1)}=M_1\beta_1$ into (\ref{ab}) and 
taking the $\beta_1$-th root, we get the equation: 
\begin{equation}
||w_{\sigma(1)}||^{2M_1}=||\sum_{|p_{\sigma(1)}|=M_1}b_{1;p_{\sigma(1)}}
(w_{\sigma(1)})^{p_{\sigma(1)}}||^2. \label{M1}
\end{equation}
Since the coefficients of 
$|(w_{\sigma(j)})^{p_{\sigma(j)}}|^2$ with $|p_{\sigma(j)}|=M_1$ 
on the left hand side is $C_{M_1, p_{\sigma(1)}}^2$ and the one on the right hand side is 
$|b^{1}_{1; p_{\sigma(1)}}|^2+\dots+|b^{n_1}_{1; p_{\sigma(1)}}|^2$, 
the vector $(1/C_{M_1,p_{\sigma(1)}})b_{1; p_{\sigma(1)}} \in \mathbb{C}^{n_1}$ 
is a unit vector. 
On the other hand, 
picking up the terms $(w_{1})^{p_{\sigma(1)}}(\ol{w_1})^{\hat{p}_{\sigma(1)}}$ with 
$p_{\sigma(1)} \ne \hat{p}_{\sigma(1)}$ from the both sides of (\ref{M1}), 
we have $<b_{1; p_{\sigma(1)}}, b_{1; \hat{p}_{\sigma(1)}}>=0$ for 
$p_{\sigma(1)} \ne \hat{p}_{\sigma(1)}$. 
Thus the set $\{(1/C_{M_1,p_{\sigma(1)}})b_{1; p_{\sigma(1)}}\}$ 
consisting of $T_{M_1, m_{\sigma(1)}}$ vectors is an orthogonal systems in 
$\mathbb{C}^{n_1}$; and hence, 
$T_{M_1, m_{\sigma(1)}} \leq n_1$.  
This is the inequality required in (a). 
For simplicity, we now denote by   
$\{U_1, \dots, U_{T_{M_1, m_{\sigma(1)}}}\}$ this orthogonal system. 
Then one can find unit vectors 
$D_1, \dots, \\ D_{n_1-T_{M_1, m_{\sigma(1)}}} \in \mathbb{C}^{n_1}$ 
in such a way that 
\begin{equation}
U=
\begin{pmatrix}
U_1 \\
\vdots \\
U_{T_{M_1, m_{\sigma(1)}}} \\
D_1 \\
\vdots \\
D_{n_1-T_{M_1, m_{\sigma(1)}}}
\end{pmatrix} \nn
\end{equation}
is an $n_1 \times n_1$ unitary matrix. 
Note that the transformation 
\begin{equation}
\Hn \ni (\wt{z}, \wt{w}_1, \dots, \wt{w}_N) \mapsto 
(\wt{z}, \wt{w}_1\ol{U}^t, \wt{w}_2, \dots, \wt{w}_N) \in \Hn  \nn
\end{equation}
is an automorphism of $\Hn$. 
Applying $U$ to $G_1$ as follows, we obtain the second part of the lemma. 
\begin{align}
G_1\ol{U}^t
&=\Bigl(\sum_{|p_{\sigma(1)}|=M_1}b^1_{1;p_{\sigma(1)}}(w_{\sigma(1)})^{p_{\sigma(1)}}, 
   \dots, \sum_{|p_{\sigma(1)}|=M_1}b^{n_1}_{1;p_{\sigma(1)}}
        (w_{\sigma(1)})^{p_{\sigma(1)}} \Bigr) \ol{U}^t  \nn \\
&=(\dots, \dfrac{C^{2}_{M_1, m_{\sigma(1)}}}{C_{M_1, p_{\sigma(1)}}}(w_{\sigma(1)})^{p_{\sigma(1)}}, 
     \dots, 0, \dots) \nn \\
&=(\dots, C_{M_1, p_{\sigma(1)}}(w_{\sigma(1)})^{p_{\sigma(1)}}, \dots, 0, \dots). \nn
\end{align}
\end{proof}

The normalization of $G_1$ in Lemma~{\ref{dangelonormal}} is obtained by 
restricting the situation to $w_{\sigma(2)}=\dots=w_{\sigma(N-1)}=w_N=0$. 
By the same reason, restricting the situation to the variety
\begin{equation}
\Bigl(\bigcap_{\substack{j=1, \dots, N-1 \\ j \ne k}}\{w_{\sigma(j)}=0\} \Bigr) \cap \{w_N=0\},  \nn
\end{equation}
we can normalize $G_k$ as 
\begin{equation}
G_k(w_{\sigma(k)})=(H_{M_k}(w_{\sigma(k)}), 0). \nn 
\end{equation}

Restricting the mapping to $w_{\sigma(1)}=\dots=w_{\sigma(N-1)}=0$ 
and tracing the same argument as above, we can normalize $G^{\lambda}_{N}$ 
by using $T_{1, m_N}$ and $C_{1, p_N}$. 
Then we conclude that $G_N$ is equivalent to $(w_N, 0)$. 

\section{proof of the main theorem}
Now we have proved that $(F, G)$ is equivalent to  
\begin{equation}
(z, H_{M_1}(w_{\sigma(1)}), 0, \dots, H_{M_{N-1}}(w_{\sigma(N-1)}), 0, w_N, 0)
\end{equation}
as a mapping between $\Hm$ and $\Hn$. 
We pull back this normalization to the mapping between 
$\Em$ and $\En$ via $\Psi$ and $\wt{\Psi}$. 
By calculation, we obtain that 
$(\mathcal{F}, \mathcal{G})$ is equivalent to 
\begin{equation}
(z, H_{M_1}(w_{\sigma(1)}), 0, \dots, H_{M_{N-1}}(w_{\sigma(N-1)}), 0, w_N, 0). \nn
\end{equation}
This finishes the proof of the Main Theorem. 

\section{Comparison with the previous result}
In this section, we compare our result with author's previous result in \cite{Ha2}. 
Notation is the same as in the present paper. 
\begin{mtTJM}
\itshape{Let $E(m;(m);(\alpha))$ and $E(n;(n);(\beta))$ be 
generalized pseudoellipsoids with $N$ blocks.  
Suppose that $2<m_j, 3<n_j$ and that 
$n-m<\mathrm{min}\{n_1, \dots, n_N\}$.    
Let $(\mathcal{F}, \mathcal{G}_1, \dots, \mathcal{G}_N) : 
E(m;(m);(\alpha)) \to E(n;(n);(\beta))$ 
be a proper holomorphic mapping that is 
holomorphic up to the boundary. 
Assume that non-zero columns of any block row in the Jacobian matrix of 
the unbounded representation 
of $(\mathcal{F}, \mathcal{G})$ are linearly independent. 
Then we have the following: 

$\mathrm{(1)}$ There exists a permutation $\sigma$ of $\{1, \dots, N\}$ such that 
$\mathcal{G}_j|_{w_{\sigma(j)}=0}=0$ for every $j$. 

$\mathrm{(2)}$ If the permutation $\sigma$ in $\mathrm{(1)}$ satisfies the inequality 
$m_{\sigma(j)} \leq n_j<2m_{\sigma(j)}-1$ for every $j$,  
then $\alpha_{\sigma(j)}=\beta_j$ for all $j$ and the proper holomorphic mapping 
$(\mathcal{F}, \mathcal{G}_1, \dots, \mathcal{G}_N)$ 
is equivalent to 
$(\widetilde{\mathcal{F}}, \widetilde{\mathcal{G}}_1, \dots, \widetilde{\mathcal{G}}_N)$ 
of the form 
\begin{equation}
\widetilde{\mathcal{F}}(z, w_1, \dots, w_N)=z, \; 
\widetilde{\mathcal{G}}_j(z, w_1, \dots, w_N)=(w_{\sigma(j)}, 0), \; 
1 \leq j \leq N. \nn
\end{equation}}
\end{mtTJM}
In \cite{Ha2}, we assume that 
$m_{\sigma(j)} \leq n_j<2m_{\sigma(j)}-1$ for every $j$. 
If we add this inequality to Main Theorem~{\ref{main}} as an assumption, 
then, combining with (b) in Main Theorem~{\ref{main}}, 
we have an inequality: 
\begin{equation}
\dfrac{(M_j+m_{\sigma(j)}-1)!}{M_j ! (m_{\sigma(j)}-1)!} \leq n_j < 2m_{\sigma(j)}-1. \label{dimineq}
\end{equation} 
Note that 
$\dfrac{(M_j+m_{\sigma(j)}-1)!}{M_j ! (m_{\sigma(j)}-1)!}$ 
is an increasing function with respect to $M_j$ and inequality (\ref{dimineq}) does not hold 
for $M_j=2$. 
Hence inequality (\ref{dimineq}) implies that $M_j=1$. 
This implies $\alpha_{\sigma(j)}=\beta_j$ in Main Theorem~{\ref{main}}, 
which is the same conclusion as 
in Main Theorem in \cite{Ha2}. 
Since the function defined by (\ref{h}) satisfies $H_1(z)=z$, 
under the assumption on dimensions 
$m_{\sigma(j)} \leq n_j < 2m_{\sigma(j)}-1$ as in \cite{Ha2}, 
our mapping $(\mathcal{F}, \mathcal{G})$ is equivalent to the mapping 
\begin{equation}
(z, w) \mapsto (z, w_{\sigma(1)}, 0, \dots, w_{\sigma(N-1)}, 0, w_N, 0), \label{totalembed}
\end{equation}
which 
is also the same conclusion in \cite{Ha2}. 
As a result, we conclude the following. 
\begin{theorem}
We add the assumption $m_{\sigma(j)} \leq n_j < 2m_{\sigma(j)}-1$ to 
Main Theorem~$\mathrm{\ref{main}}$. 
Then it implies Main Theorem in \cite{Ha2}. 
\end{theorem}
Since $M_j$ is invariant under the composition of automorphisms of $\Em$ and $\En$, 
the mapping 
\begin{equation}
(z, w) \mapsto 
(z, H_{M_1}(w_{\sigma(1)}), 0, \dots, H_{M_{N-1}}(w_{\sigma(N-1)}), 0, w_N, 0) \nn
\end{equation}
with $M_j \not=1$ is not equivalent to the mapping (\ref{totalembed}). 
Finally we would like to refer to the assumption on the Jacobian matrix in 
Main Theorem in \cite{Ha2}. 
Let $\mathcal{G}_j$ be a component of a proper holomorphic mapping as in 
Main Theorem~${\ref{main}}$. 
The assumption on the Jacobian matrix in Main Theorem in \cite{Ha2} means that 
the non-zero columns of the Jacobian matrix of $\mathcal{G}_j$ 
are linearly independent for any $j=1, \dots, N-1$. 
This condition is satisfied for $\mathcal{G}_j(w_{\sigma(j)})=(H_{M_j}(w_{\sigma(j)}), 0)$.

\end{document}